\documentclass[aop,MSNbibl,seceqn,noMRlinks,dvips]{arximspdf}

\doi{10.1214/08-AOP400}
\volume{37}
\issue{1}
\pubyear{2009}
\firstpage{189}
\lastpage{205}

\makeatletter
\def\epsilon{\varepsilon}

\newproclaim{definition}{Definition}
\newproclaim{remark}{Remark}
\newtheorem{proposition}{Proposition}
\newtheorem{lemma}{Lemma}

\makeatother

\begin{document}
\begin{frontmatter}

\title{Large deviations for random walk in a space--time product
environment\thanksref{TIT1}}
\runtitle{Large deviations for RW in a space--time PE}
\thankstext{TIT1}{Supported in part by a NSF Grant DMS-06-04380.}

\begin{aug}
\author[A]{\fnms{Atilla} \snm{Yilmaz}\ead[label=e1]{yilmaz@cims.nyu.edu}
\ead[label=u1,url]{http://www.math.nyu.edu/\textasciitilde yilmaz/}\corref{}}
\address[A]{Courant Institute\\  251 Mercer Street\\  New York, New York
10012-1185\\  USA\\  \printead{e1}\\  \printead{u1}}
\affiliation{Courant Institute of Mathematical Sciences}
\runauthor{A. Yilmaz}
\end{aug}

\received{\smonth{11} \syear{2007}}

%
\begin{abstract}
We consider random walk $(X_n)_{n\geq0}$ on $\mathbb{Z}^d$ in a
space--time product environment $\omega\in\Omega$. We take the point of
view of the particle and focus on the environment Markov chain
$(T_{n,X_n}\omega)_{n\geq0}$ where $T$ denotes the shift on~$\Omega$.
Conditioned on the particle having asymptotic mean velocity
equal to any given $\xi$, we show that the empirical process of the
environment Markov chain converges to a stationary process
$\mu_\xi^\infty$ under the averaged measure. When $d\geq3$ and $\xi$ is
sufficiently close to the typical velocity, we prove that averaged and
quenched large deviations are equivalent and when conditioned on the
particle having asymptotic mean velocity $\xi$, the empirical process
of the environment Markov chain converges to $\mu_\xi^\infty$ under the
quenched measure as well. In this case, we show that $\mu_\xi^\infty$
is a stationary Markov process whose kernel is obtained from the
original kernel by a Doob $h$-transform.
\end{abstract}

%

\end{frontmatter}

\section{Introduction}

Random walk in a random environment (RWRE) is one of the standard
models in the study of random media. It aims to capture the essence of
the motion of a particle in a disordered medium. Mathematically, it is
a discrete time Markov chain on $\mathbb{Z}^d$ with random transition
probabilities which are sampled from a given joint distribution and
kept fixed throughout the walk. See Sznitman \cite{Sznitman} or
Zeitouni \cite{Zeitouni} for a survey.

Instead, if we assume that the transition probabilities at distinct
states are i.i.d. and are freshly sampled at each time step, we get what
is known as random walk in a space--time product environment. Here is
the precise formulation: For each $n\in\mathbb{Z}$ and $x,y\in\mathbb
{Z}^d$, we write $\pi_{n,n+1}(x,x+y)$ to denote the random probability
of the particle being at position $x+y$ at time $n+1$ given it is at
position $x$ at time $n$. We refer to the random probability vector
$\omega_{n,x}:= (\pi_{n,n+1}(x,x+y) )_{y\in\mathbb{Z}^d}$ as the
environment at position $x$ at time $n$ and assume the environment
$\omega:= (\omega_{n,x} )_{n\in\mathbb{Z},x\in\mathbb{Z}^d}$ to be an
i.i.d. collection. The environments form a probability space $ (\Omega
,\mathcal{B},\mathbb{P} )$ where $\mathcal{B}$ is the Borel $\sigma
$-algebra on $\Omega$ and $\mathbb{P}$ is a product measure. For every
$n\in\mathbb{Z}$, we write $\mathcal{B}_n^+$ and $\mathcal{B}_n^-$ to
mean the $\sigma$-algebras generated by $ (\omega_{m,x}\dvtx x\in\mathbb
{Z}^d, m\geq n )$ and $ (\omega_{m,x}\dvtx x\in\mathbb{Z}^d, m\leq n )$,
respectively.

Given $\omega\in\Omega$, the Markov chain $ (X_n )_{n\geq k}$ starting
at position $x\in\mathbb{Z}^d$ at time $k\in\mathbb{Z}$ induces a
measure $P_{k,x}^{\omega}$ on the space of paths, called the
``quenched'' measure. The semi-direct product $P_{k,x}:=\mathbb{P}\times
P_{k,x}^{\omega}$ is referred to as the ``averaged'' measure.
Expectations with respect to $\mathbb{P}, P_{k,x}^{\omega}$ and
$P_{k,x}$ are denoted by $\mathbb{E}, E_{k,x}^{\omega}$ and $E_{k,x}$,
respectively. For convenience, we sometimes write $Z_{n+1}$ to mean
$X_{n+1}-X_n$.

We note that if $ (X_n )_{n\geq0}$ is random walk on $\mathbb{Z}^d$ in
a space--time product environment, then $ (n,X_n )_{n\geq0}$ can be
viewed as RWRE on $\mathbb{Z}^{d+1}$.

We set $U := \{z\in\mathbb{Z}^d\dvtx |z|=1 \}$. In order to provide short
proofs, we assume that the walk is nearest-neighbor and the environment
is uniformly elliptic, that is, $\mathbb{P}(\pi_{0,1}(0,z)>0)=0$ unless
$z\in U$, and there exists a constant $c>0$ such that $\mathbb{P}(\pi
_{0,1}(0,z)\geq c)=1$ for $z\in U$.

We define the shifts $ (T_{m,y} )_{m\in\mathbb{Z}, y\in\mathbb{Z}^d}$
on $\Omega$ by $ (T_{m,y}\omega)_{n,x}=\omega_{n+m,x+y}$. Given a
random path $ (X_n )_{n\geq0}$, we consider $ (T_{n,X_n}\omega
)_{n\geq0}$ which is a Markov process with state space $\Omega$. It is
referred to as the ``environment Markov chain'' and its transition
kernel is given by $\overline{\pi}(\omega,T_{1,z}\omega):=\pi
_{0,1}(0,z)$ for every $z\in U$. What it does is simply observe the
environment from the point of view of the particle. This is a standard
approach in the study of random media. See, for example, De Masi et al.
\cite{DeMasi}, Kipnis and Varadhan \cite{KV}, Kozlov \cite{Kozlov},
Olla \cite{Olla} or Papanicolaou and Varadhan \cite{PV}.

By a generalization of a technique first given in Kozlov \cite{Kozlov},
Rassoul-Agha \cite{Firas} shows that the environment Markov chain has
a unique invariant measure that is absolutely continuous relative to
$\mathbb{P}$ on every $\mathcal{B}_n^+$.

In Section \ref{Asection}, we focus on the averaged measure. The
marginal of $P_{o,o}$ on paths is classical random walk with transition
vector $ (q(z) )_{z\in U}$ given by $q(z)=\mathbb{E}[\pi_{0,1}(0,z)]$
for every $z\in U$. Therefore, the law of large numbers (LLN) for the
mean velocity of the particle under $P_{o,o}$ is valid and the limiting
velocity is $\xi_o:=\sum_{z\in U}q(z)z$. The averaged large deviation
principle (LDP) for the mean velocity of the particle is simply Cram\'
{e}r's theorem (see, e.g., Dembo and Zeitouni \cite{DZ}) and the rate
function $I_a$ is the convex conjugate of the logarithm of the moment
generating function $\phi$, given by
%
\begin{equation}
\phi(\theta)=\sum_{z\in U}q(z)e^{\langle\theta,z\rangle}.\label{fi}
\end{equation}
We set $\mathcal{D}:= \{\xi\in\mathbb{R}^d \dvtx I_a(\xi)<\infty\}$.
Given $\xi\in\mathrm{int} (\mathcal{D} )$, we consider the event
defined by the particle having asymptotic mean velocity $\xi$. If $\xi
\neq\xi_o$, this is a rare event and the exponential rate of decay of
its $P_{o,o}$-probability is given by $I_a(\xi)>0$. Conditioned on this
event, we expect the environment Markov chain to behave differently.
Indeed, we show that the empirical process of the environment Markov
chain under this conditioning converges to a stationary process
uniquely determined by $\xi$. Here is how we accomplish this: We first
give a definition.
\begin{definition}\label{definemu}
For every $\xi\in\mathrm{int} (\mathcal{D} )$, we define a measure
$\overline{\mu}_\xi^\infty$ on $\Omega\times U^\mathbb{N}$ in the
following way: There exists a unique $\theta\in\mathbb{R}^d$ satisfying
$\xi=\nabla\log\phi(\theta)$. For every $N,M$ and $K\in\mathbb{N}$, we
take any bounded function $f\dvtx \Omega\times U^{\mathbb{N}}\rightarrow
\mathbb{R}$ such that $f(\cdot,(z_i)_{i\geq1})$ is independent of
$(z_i)_{i>K}$ and $\mathcal{B}_{-N}^+\cap\mathcal{B}_M^-$-measurable
for each $(z_i)_{i\geq1}$. Then
\begin{eqnarray}\label{mucan}
&&\int fd\overline{\mu}_\xi^\infty\nonumber\\[-8pt]\\[-8pt]
&&\qquad:=E_{o,o} \bigl[e^{\langle\theta,X_{N+M+K+1}\rangle- (N+M+K+1)\log
\phi(\theta)}f(T_{N,X_N}\omega,(Z_{N+i})_{i\geq1}) \bigr]\nonumber
\end{eqnarray}
where $X_i$ and $Z_i=X_i-X_{i-1}$ are as defined earlier.
\end{definition}

In Proposition \ref{welldefined}, we show that $\overline{\mu}_\xi
^\infty$ is well defined. In Proposition \ref{stationary}, we show that
$\overline{\mu}_\xi^\infty$ induces a stationary process $\mu_\xi^\infty
$ with values in $\Omega$. Let us define the events that we use in the
statement of the main theorem of Section \ref{Asection}.
\begin{definition}\label{events}
For every $\xi\in\mathrm{int} (\mathcal{D} )$, $N,M$ and $K\in\mathbb
{N}$, $f$ as in Definition \ref{definemu}, $\epsilon>0$ and $n\in\mathbb
{N}$, we define the event
%
\begin{equation}\label{A}
A_{\xi,n}^\epsilon(f):= \Biggl\{ \Biggl|\frac{1}{n}\sum_{j=0}^{n-1}f(T_{j,X_j}\omega
,(Z_{j+i})_{i\geq1})-\int f\,d\overline{\mu}_\xi^\infty
\Biggr|>\epsilon\Biggr\}.
\end{equation}
Given $\delta>0$, we define the event
%
\begin{equation}\label{D}
D_{\xi,n}^\delta:= \biggl\{ \biggl|\frac{X_n}{n}-\xi\biggr|\leq\delta\biggr\}.
\end{equation}
\end{definition}

Finally, we prove the following theorem.

\begin{thm}\label{averagedconditioning}
For every $\xi\in\mathrm{int} (\mathcal{D} )$, $N,M$ and $K\in\mathbb
{N}$, $f$ as in Definition \ref{definemu} and $\epsilon>0$, there
exists $\delta_o>0$ such that for every $\delta>0$ with $\delta<\delta
_o$
\[
\limsup_{n\rightarrow\infty}\frac{1}{n}\log P_{o,o} (A_{\xi,n}^\epsilon
(f) |D_{\xi,n}^\delta )<0
\]
where the events $A_{\xi,n}^\epsilon(f)$ and $D_{\xi,n}^\delta$ are
defined in (\ref{A}) and (\ref{D}), respectively.
\end{thm}

In Section \ref{Qsection}, we focus on the quenched measure. Varadhan
\cite{Raghu} proves the quenched LDP for the mean velocity of the
particle for the related model of RWRE. In our case, even though we can
think of $(n,X_n)_{n\geq0}$ as RWRE on $\mathbb{Z}^{d+1}$, Varadhan's
result is not directly applicable since he assumes that the environment
is uniformly elliptic, which we do not have in the ``time'' direction.
However, one expects that a modification of his argument should work.
Instead of taking this route, we develop an alternative technique and prove
the following theorem.
\begin{thm}\label{AequalsQ}
If $d\geq3$, then there exists $\eta>0$ such that the quenched LDP for
the mean velocity of the particle holds in the $\eta$-neighborhood of
$\xi_o$ and the rate function is identically equal to the rate function
$I_a$ of the averaged LDP in this neighborhood.
\end{thm}

\begin{remark}
Theorem \ref{AequalsQ} is similar in flavor to the results of Flury
\cite{Flury}, Song and Zhou \cite{Song} and Zygouras \cite{Nikos} on
the related model of random walk with a random potential. For
multidimensional random walk in a product environment, Varadhan \cite
{Raghu} proves that the averaged LDP also holds and the corresponding
rate function has the same zero set with the quenched rate function.
However, it is not known whether the two rate functions agree on a
neighborhood of the LLN velocity. The case of RWRE on $\mathbb{Z}$ is
closely studied by Comets, Gantert and Zeitouni \cite{CGZ} who in
particular show that there is no neighborhood of the LLN velocity on
which the two rate functions agree.
\end{remark}

Once again, conditioned on the particle having asymptotic mean velocity
$\xi$, we can ask what the empirical process of the environment Markov
chain converges to, but this time under the quenched measure. Whenever
the quenched LDP for the mean velocity holds in a neighborhood of $\xi$
with rate $I_a(\xi)$ at $\xi$---in particular when $d\geq3$ and $|\xi
-\xi_o|<\eta$---one expects that the answer should be again $\mu_\xi
^\infty$. The following theorem is a result to this effect.

\begin{thm}\label{strong}
For every $\xi\in\mathrm{int} (\mathcal{D} )$, if the quenched LDP for
the mean velocity holds in a neighborhood of $\xi$ with rate $I_a(\xi)$
at $\xi$, then
for every $N,M$ and $K\in\mathbb{N}$, $f$ as in Definition \ref
{definemu} and $\epsilon>0$, there exists $\delta_o>0$ such that for
every $\delta>0$ with $\delta<\delta_o$, $\mathbb{P}$-a.s.
\[
\limsup_{n\rightarrow\infty}\frac{1}{n}\log P_{o,o}^\omega(A_{\xi
,n}^\epsilon(f) |D_{\xi,n}^\delta )<0,
\]
where the events $A_{\xi,n}^\epsilon(f)$ and $D_{\xi,n}^\delta$ are
defined in (\ref{A}) and (\ref{D}), respectively.
\end{thm}

Finally, in Section \ref{Doobsection}, we reveal the structure of $\mu
_\xi^\infty$.

\begin{thm}\label{doob}
For $d\geq3$ and $|\xi-\xi_o|<\eta$ with $\eta$ as in Theorem \ref
{AequalsQ}, we let $\theta\in\mathbb{R}^d$ be the unique solution of
$\xi=\nabla\log\phi(\theta)$. There exists a $\mathcal
{B}_o^+$-measurable function $u^\theta>0$ which satisfies $\int u^\theta
\,d\mathbb{P} = 1$ and $\mathbb{P}$-a.s.
\[
u^\theta(\omega)=\sum_{z\in U}\overline{\pi}(\omega,T_{1,z}\omega
)e^{\langle\theta,z\rangle-\log\phi(\theta)}u^\theta
(T_{1,z}\omega).
\]
We define a transformed kernel $\overline{\pi}^\theta$ on $\Omega$ by
\[
\overline{\pi}^\theta(\omega,T_{1,z}\omega) = \overline{\pi}(\omega
,T_{1,z}\omega)\frac{u^\theta(T_{1,z}\omega)}{u^\theta(\omega)}
e^{\langle\theta,z\rangle-\log\phi(\theta)}.
\]
$\mu_\xi^\infty$ is the unique stationary Markov process with
transition kernel $\overline{\pi}^\theta$ and whose marginal $\mu_\xi
^1$ is absolutely continuous relative to $\mathbb{P}$ on every $\mathcal
{B}_n^+$. 
\end{thm}

In other words, conditioned on the particle having asymptotic mean
velocity $\xi$, the environment Markov chain chooses to switch from
kernel $\overline{\pi}$ to kernel $\overline{\pi}^\theta$. The most
economical tilt in terms of large deviations is given by a Doob $h$-transform.

\section{Conditioning under the averaged measure}\label{Asection}

For any $n\in\mathbb{N}$ and $\theta\in\mathbb{R}^d$, since the
environment is i.i.d.,
%
\begin{equation}
E_{o,o} \bigl[e^{\langle\theta,X_n\rangle} \bigr]=\phi(\theta)^n\label{budur}
\end{equation}
with the notation in (\ref{fi}). By Cram\'{e}r's theorem, the LDP for
the mean velocity of the particle holds under $P_{o,o}$ with the rate
function $I_a$ given by
%
\begin{equation}\label{Ia}
I_a(\xi)=\sup_{\theta'} \{\langle\theta',\xi\rangle-\log\phi(\theta')
\} = \langle\theta,\xi\rangle-\log\phi(\theta),
\end{equation}
where $\theta$ is the unique solution of $\xi=\nabla\log\phi(\theta)$.
Due to our nearest-neighbor and ellipticity assumptions,
\[
\mathcal{D}= \{\xi\in\mathbb{R}^d \dvtx I_a(\xi)<\infty\}= \{(\xi^1,\ldots
,\xi^d)\in\mathbb{R}^d \dvtx|\xi^1|+\cdots+|\xi^d|\leq1 \}.
\]

\begin{proposition}\label{welldefined}
For every $\xi\in\mathrm{int} (\mathcal{D} )$, the measure $\overline
{\mu}_\xi^\infty$ on $\Omega\times U^\mathbb{N}$, given in Definition
\ref{definemu}, is well defined.
\end{proposition}

\begin{pf}
For every $N,M$ and $K\in\mathbb{N}$, we consider any $f$ as in
Definition \ref{definemu}. We set $L:=N+M+K+1$. Since $f(\cdot
,(z_i)_{i\geq1})$ is independent of $(z_i)_{i>K}$ and $\mathcal
{B}_{-N}^+\cap\mathcal{B}_M^-$-measurable for each $(z_i)_{i\geq1}$, we
see that for every $N',M'$ and $K'\in\mathbb{N}$ with $N\leq N'$,
$M\leq M'$ and $K\leq K'$, $f(\cdot,(z_i)_{i\geq1})$ is independent of
$(z_i)_{i>K'}$ and $\mathcal{B}_{-N'}^+\cap\mathcal
{B}_{M'}^-$-measurable for each $(z_i)_{i\geq1}$ as well. So, we need
to show that (\ref{mucan}) does not change if we replace $N$ by $N+1$,
$M$ by $M+1$, or $K$ by $K+1$.

Let us start with the argument for $N$. We observe that
%
\begin{eqnarray}\label{bagimsiz}
&&E_{o,o} \bigl[e^{\langle\theta,X_{L+1}\rangle-(L+1)\log\phi(\theta
)}f(T_{N+1,X_{N+1}}\omega,(Z_{N+1+i})_{i\geq1}) \bigr]\nonumber\\
&&\qquad=\sum_{x}\mathbb{E} \bigl(E_{o,o}^\omega\bigl[e^{\langle\theta
,X_1\rangle-\log\phi(\theta)}, X_1=x \bigr] \\
&&\quad\qquad\hspace*{25pt}
{}\times E_{1,x}^\omega\bigl[e^{\langle\theta,X_{L+1}-X_1\rangle
-L\log\phi(\theta)}f(T_{N+1,X_{N+1}}\omega,(Z_{N+1+i})_{i\geq1}) \bigr]
\bigr)\nonumber\\
&&\qquad=\sum_{x}E_{o,o} \bigl[e^{\langle\theta,X_1\rangle-\log
\phi(\theta)}, X_1=x \bigr]\label{gozum}\\
&&\qquad\quad\hspace*{13pt}{}    \times E_{1,x} \bigl[e^{\langle\theta,X_{L+1}-X_1\rangle-L\log\phi(\theta
)}f(T_{N+1,X_{N+1}}\omega,(Z_{N+1+i})_{i\geq1}) \bigr]\nonumber\\
&&\qquad=\sum_{x}E_{o,o} \bigl[e^{\langle\theta,X_1\rangle-\log\phi(\theta
)}, X_1=x \bigr]\label{kaydir}\\
&&\quad\qquad\hspace*{13pt}{}    \times E_{o,o} \bigl[e^{\langle
\theta,X_L\rangle-L\log\phi(\theta)}f(T_{N,X_N}\omega,(Z_{N+i})_{i\geq
1}) \bigr]\nonumber\\
&&\qquad=E_{o,o} \bigl[e^{\langle\theta,X_L\rangle-L\log\phi(\theta
)}f(T_{N,X_N}\omega,(Z_{N+i})_{i\geq1}) \bigr]\nonumber
\end{eqnarray}
holds. We note that each term of the sum in (\ref{bagimsiz}) is the
$\mathbb{P}$-expectation of two random variables; the first one is
$\mathcal{B}_0^-$-measurable and the second one is $\mathcal
{B}_1^+$-measurable. We make use of this independence to obtain (\ref
{gozum}). We also note that we use the stationarity of $\mathbb{P}$
under shifts to obtain (\ref{kaydir}) from (\ref{gozum}). Hence, (\ref
{mucan}) does not change if we replace $N$ by $N+1$.

Similarly, if we replace $M$ by $M+1$ in (\ref{mucan}),
\begin{eqnarray} \label{baris}
&&E_{o,o} \bigl[e^{\langle\theta,X_{L+1}\rangle-(L+1)\log\phi(\theta
)}f(T_{N,X_N}\omega,(Z_{N+i})_{i\geq1}) \bigr]\nonumber\\
&&\qquad=\sum_{x}\mathbb{E} \bigl(E_{o,o}^\omega\bigl[e^{\langle\theta
,X_L\rangle-L\log\phi(\theta)}f(T_{N,X_N}\omega,(Z_{N+i})_{i\geq1}),
X_L=x \bigr]
\\
&&\qquad\quad\hspace*{136pt}{}       \times E_{L,x}^\omega\bigl[
e^{\langle\theta,X_{L+1}-X_L\rangle-\log\phi(\theta)} \bigr]
\bigr)\nonumber\\
&&\qquad=\sum_{x}E_{o,o} \bigl[e^{\langle\theta,X_L\rangle-L\log\phi(\theta
)}f(T_{N,X_N}\omega,(Z_{N+i})_{i\geq1}), X_L=x \bigr]\nonumber\\
&&\qquad\quad\phantom{\sum_{x}}{}    \times E_{L,x} \bigl[e^{\langle\theta,X_{L+1}-X_L\rangle-\log\phi
(\theta)} \bigr]\nonumber\\
&&\qquad=E_{o,o} \bigl[e^{\langle\theta,X_L\rangle-L\log\phi(\theta
)}f(T_{N,X_N}\omega,(Z_{N+i})_{i\geq1}) \bigr]\nonumber
\end{eqnarray}
where, once again, we use the fact that each term of the sum in (\ref
{baris}) is the $\mathbb{P}$-expectation of two random variables; the
first one is $\mathcal{B}_{L-1}^-$-measurable and the second one is
$\mathcal{B}_L^+$-measurable.

We finally note that the argument for $K$ is the same as the one for $M$.
\end{pf}

\begin{proposition}\label{stationary}
$\overline{\mu}_\xi^\infty$ induces a stationary process $\mu_\xi^\infty
$ with values in $\Omega$.
\end{proposition}

\begin{pf}
We define $\bar{S}\dvtx \Omega\times U^{\mathbb{N}}\rightarrow\Omega\times
U^{\mathbb{N}}$ by
\[
\bar{S}\dvtx (\omega,(z_i)_{i\geq1} )\mapsto (T_{1,z_1}\omega,(z_i)_{i\geq
2} )
\]
and the projection map $\Psi\dvtx\Omega\times U^{\mathbb{N}}\rightarrow
\Omega$ by $\Psi\dvtx (\omega,(z_i)_{i\geq1} )\mapsto\omega.$ Let us show
that $\overline{\mu}_\xi^\infty$ is invariant under $\bar{S}$:

For every $N,M$ and $K\in\mathbb{N}$, $f$ as in Definition \ref
{definemu} and $(z_i)_{i\geq1}$, we see that $f\circ\bar{S} (\omega
,(z_i)_{i\geq1} )=f (T_{1,z_1}\omega,(z_i)_{i\geq2} )$ is $\mathcal
{B}_{-(N-1)}^+\cap\mathcal{B}_{M+1}^-$-measurable and independent of
$(z_i)_{i>K+1}$. By definition,
\begin{eqnarray*}
&&\int f\circ\bar{S}\,d\overline{\mu}_\xi^\infty\\
&&\qquad=E_{o,o} \bigl[e^{\langle\theta,X_{N+M+K+2}\rangle- (N+M+K+2)\log
\phi(\theta)}\\
&&\qquad\phantom{=E_{o,o} \bigl[}
{}\times f\circ\bar{S}\bigl(T_{N-1,X_{N-1}}\omega,\bigl(Z_{(N-1)+i}\bigr)_{i\geq
1}\bigr) \bigr]\\
&&\qquad=E_{o,o} \bigl[e^{\langle\theta,X_{N+M+K+2}\rangle-
(N+M+K+2)\log\phi(\theta)}f(T_{N,X_N}\omega,(Z_{N+i})_{i\geq1}) \bigr]\\
&&\qquad=\int f\,d\overline{\mu}_\xi^\infty.
\end{eqnarray*}

Therefore, under $\overline{\mu}_\xi^\infty$, $ (\Psi\circ\bar
{S}^k(\cdot) )_{k\geq0}$ extends to a stationary process with values in
$\Omega$, whose distribution we denote by $\mu_\xi^\infty$.
\end{pf}

\begin{pf*}{Proof of Theorem \ref{averagedconditioning}}
Under the hypotheses of the theorem, we take any $\epsilon>0$, $\delta
>0$ and recall (\ref{A}) and (\ref{D}). We define the event
%
\begin{equation}\label{A+}
A_{\xi,n}^{\epsilon,+}(f):= \Biggl\{\frac{1}{n}\sum
_{j=0}^{n-1}f(T_{j,X_j}\omega,(Z_{j+i})_{i\geq1})-\int f\,d\overline{\mu}_\xi^\infty>\epsilon\Biggr\}.
\end{equation}
We similarly define $A_{\xi,n}^{\epsilon,-}(f)$ and have $A_{\xi
,n}^\epsilon(f)=A_{\xi,n}^{\epsilon,+}(f)\cup A_{\xi,n}^{\epsilon
,-}(f)$. So, it suffices to prove Theorem \ref{averagedconditioning}
for only, say, $A_{\xi,n}^{\epsilon,+}(f)$.

To simplify the notation, we write
%
\begin{equation}
F_j := f(T_{j,X_j}\omega,(Z_{j+i})_{i\geq1})-\int f\,d\overline
{\mu}_\xi^\infty.\label{F}
\end{equation}
Then (\ref{A+}) becomes
\[
A_{\xi,n}^{\epsilon,+}(f)= \Biggl\{\frac{1}{n}\sum_{j=0}^{n-1}F_j>\epsilon\Biggr\}.
\]

Since $\xi\in\mathrm{int} (\mathcal{D} )$, the unique solution $\theta$
of $\xi=\nabla\log\phi(\theta)$ satisfies
\[
I_a(\xi)=\langle\theta,\xi\rangle-\log\phi(\theta).
\]
By a standard change of measure argument and the averaged LDP, we see
that for any $z>0$,
\begin{eqnarray}\label{cebi}
&&\limsup_{n\rightarrow\infty}\frac{1}{n}\log P_{o,o} (A_{\xi
,n}^{\epsilon,+}(f) |D_{\xi,n}^\delta )\nonumber\\
&&\qquad\leq \limsup
_{n\rightarrow\infty}\frac{1}{n}\log P_{o,o} (A_{\xi,n}^{\epsilon
,+}(f),D_{\xi,n}^\delta)-\liminf_{n\rightarrow\infty}\frac{1}{n}\log
P_{o,o} (D_{\xi,n}^\delta)\nonumber\\
&&\qquad\leq \limsup_{n\rightarrow\infty
}\frac{1}{n}\log E_{o,o} \biggl[e^{\langle\theta,X_n\rangle},A_{\xi
,n}^{\epsilon,+}(f),
\biggl|\frac{X_n}{n}-\xi\biggr|\leq\delta\biggr]\nonumber\\[-8pt]\\[-8pt]
&&\quad\qquad{} -
\langle\theta,\xi\rangle+ I_a(\xi) + |\theta|\delta\nonumber\\
&&\qquad{}\leq
\limsup_{n\rightarrow\infty}\frac{1}{n}\log E_{o,o} \bigl[
e^{\langle\theta,X_n\rangle-n\log\phi(\theta)},A_{\xi,n}^{\epsilon
,+}(f) \bigr] + |\theta|\delta\nonumber\\
&&\qquad\leq \limsup_{n\rightarrow\infty
}\frac{1}{n}\log E_{o,o} \bigl[e^{\langle\theta,X_n\rangle-n\log\phi
(\theta) + z\sum_{j=0}^{n-1}F_j} \bigr] - z\epsilon+ |\theta|\delta,\nonumber
\end{eqnarray}
where the last line is obtained by Chebyshev's inequality. We set $L :=
N+M+K+1$ and note that
\begin{eqnarray}\label{carpim}
&&E_{o,o} \bigl[e^{\langle\theta,X_n\rangle-n\log\phi(\theta) + z\sum
_{j=0}^{n-1}F_j} \bigr]\nonumber\\[-8pt]\\[-8pt]
&&\qquad\leq \prod_{i=0}^{L-1}E_{o,o} \bigl[
e^{\langle\theta,X_n\rangle-n\log\phi(\theta) +
Lz(F_i+F_{L+i}+F_{2L+i}+\cdots)} \bigr]^{{1}/{L}}\nonumber
\end{eqnarray}
holds by an application of H\"{o}lder's inequality under
$e^{\langle\theta,X_n\rangle-n\log\phi(\theta)}\,dP_{o,o}$.

For any $i\in\{0,\ldots,L-1\}$, we let $k=k(i)$ be the largest integer
such that $kL+i<n$. Then for $n\geq2L,$
\begin{eqnarray*}
&&E_{o,o} \bigl[e^{\langle\theta,X_n\rangle-n\log\phi(\theta) +
Lz(F_i+\cdots+F_{(k-1)L+i}+F_{kL+i})} \bigr]\\
&&\qquad= \sum_{x}\mathbb{E} \bigl(
E_{o,o}^\omega\bigl[e^{\langle\theta,X_{kL+i-N}\rangle-(kL+i-N)\log
\phi(\theta) + Lz(F_i+\cdots+F_{(k-1)L+i})},\\
&&\hspace*{240pt}
X_{kL+i-N}=x \bigr] \\
&&\hspace*{27pt}\qquad\quad{}\times E_{kL+i-N,x}^\omega\bigl[e^{\langle\theta
,X_n-X_{kL+i-N}\rangle-(n-(kL+i-N))\log\phi(\theta) + Lz(F_{kL+i})} \bigr] \bigr).
\end{eqnarray*}
Each term of the above sum is the $\mathbb{P}$-expectation of the
product of two random variables and these variables are independent
since for every $(z_i)_{i\geq1}$, $f(\cdot,(z_i)_{i\geq1})$ is $\mathcal
{B}_{-N}^+\cap\mathcal{B}_M^-$-measurable and independent of
$(z_i)_{i>K}$. Using this and the fact that $\mathbb{P}$ is invariant
under shifts, we write
\begin{eqnarray*}
&&E_{o,o} \bigl[e^{\langle\theta,X_n\rangle-n\log\phi(\theta) +
Lz(F_i+\cdots+F_{(k-1)L+i}+F_{kL+i})} \bigr]\\
&&\qquad=\sum_{x}E_{o,o} \bigl[e^{\langle\theta,X_{kL+i-N}\rangle-(kL+i-N)\log\phi(\theta) +
Lz(F_i+\cdots+F_{(k-1)L+i})},X_{kL+i-N}=x \bigr]\\
&&\qquad\quad\hspace*{13pt}{}\times E_{o,o} \bigl[
e^{\langle\theta,X_{n-(kL+i-N)}\rangle-(n-(kL+i-N))\log\phi(\theta) +
LzF_N} \bigr]\\
&&\qquad=E_{o,o} \bigl[e^{\langle\theta,X_{kL+i-N}\rangle
-(kL+i-N)\log\phi(\theta) + Lz(F_i+\cdots+F_{(k-1)L+i})} \bigr]\\
&&\quad\qquad{}\times
E_{o,o} \bigl[e^{\langle\theta,X_{n-(kL+i-N)}\rangle
-(n-(kL+i-N))\log\phi(\theta) + LzF_N} \bigr].
\end{eqnarray*}
Iterating this, we get
\begin{eqnarray*}
&&E_{o,o} \bigl[e^{\langle\theta,X_n\rangle-n\log\phi(\theta) +
Lz(F_i+\cdots+F_{(k-1)L+i}+F_{kL+i})} \bigr]\\
&&\qquad=E_{o,o} \bigl[e^{\langle
\theta,X_{L+i-N}\rangle-(L+i-N)\log\phi(\theta) + LzF_i} \bigr]\\
&&\quad\qquad{}\times
E_{o,o} \bigl[e^{\langle\theta,X_L\rangle-L\log\phi(\theta) +
LzF_N} \bigr]^{k-1}\\
&&\quad\qquad{}\times E_{o,o} \bigl[e^{\langle\theta
,X_{n-(kL+i-N)}\rangle-(n-(kL+i-N))\log\phi(\theta) + LzF_N} \bigr].
\end{eqnarray*}
Since $n-(kL+i-N)\leq L+N <\infty$ and $f$ is bounded, the first and
the last terms of the above product are bounded and
\begin{eqnarray*}
&&\lim_{n\rightarrow\infty}\frac{1}{n}\log E_{o,o} \bigl[e^{\langle
\theta,X_n\rangle-n\log\phi(\theta) + Lz(F_i+\cdots
+F_{(k-1)L+i}+F_{kL+i})} \bigr]\\
&&\qquad= \frac{1}{L}\log E_{o,o} \bigl[
e^{\langle\theta,X_L\rangle-L\log\phi(\theta) + LzF_N} \bigr].
\end{eqnarray*}
Recalling (\ref{carpim}), we now know that
\begin{eqnarray*}
&&\limsup_{n\rightarrow\infty}\frac{1}{n}\log E_{o,o} \bigl[
e^{\langle\theta,X_n\rangle-n\log\phi(\theta) + z\sum_{j=0}^{n-1}F_j}
\bigr]\\
&&\qquad\leq \frac{1}{L}\log E_{o,o} \bigl[e^{\langle\theta,X_L\rangle
-L\log\phi(\theta) + LzF_N} \bigr]\\
&&\qquad=: \zeta(z).
\end{eqnarray*}
Because of (\ref{cebi}), to conclude the proof, it suffices to show
that $\zeta(z)=o(z)$. But $\zeta(0)=0$ and we recall (\ref{F}) to see
that
\[
\zeta'(0)=E_{o,o} \bigl[e^{\langle\theta,X_L\rangle-L\log\phi(\theta
)}F_N \bigr]=0
\]
precisely by Definition \ref{definemu}. Hence, we are done.
\end{pf*}

\section{Conditioning under the quenched measure}\label{Qsection}

In this section, we obtain the function $u^\theta$ mentioned in the
statement of Theorem \ref{doob}, derive some of its properties and
define the transformed kernel $\overline{\pi}^\theta$ also mentioned in
Theorem \ref{doob}. Finally, having built the necessary machinery, we
prove Theorems \ref{AequalsQ} and \ref{strong}.

\subsection{The main estimate}

In the rest of the article, the following family of functions play a
central role:

\begin{definition}
For every $\theta\in\mathbb{R}^d$, $x\in\mathbb{Z}^d$ and $n,N\in\mathbb
{Z}$ with $n<N$, we define
%
\begin{equation}\label{babalar}
u_N^\theta(\omega,n,x):=\frac{E_{n,x}^\omega[e^{\langle\theta
,X_N-X_n\rangle} ]}{\phi(\theta)^{N-n}}
\end{equation}
where $\phi$ is given in (\ref{fi}).
\end{definition}

The main estimate that enables us to obtain $u^\theta$ and establish
the equivalence of quenched and averaged large deviations is given as
\begin{lemma}\label{eltu}
If $d\geq3$, then there exists $\bar{\eta}>0$ such that for every
$\theta\in\mathbb{R}^d$ with $|\theta|<\bar{\eta}$, $x\in\mathbb{Z}^d$
and $n\in\mathbb{Z}$, we have
\[
\sup_{N>n} \|u_N^\theta(\cdot,n,x) \|_{L^2(\mathbb{P})}<\infty.
\]
\end{lemma}

\begin{pf}
It suffices to prove the lemma for $n=0$ and $x=0$.
%
\begin{eqnarray}\label{aboo}
G_N(\theta)&:=& \|u_N^\theta(\cdot,0,0) \|_{L^2(\mathbb{P})}^2=\frac
{\mathbb{E} (E_{o,o}^\omega[e^{\langle\theta,X_N\rangle} ]^2
)}{\phi(\theta)^{2N}}\\
&\hspace*{3pt}=&\mathop{\sum_{x_o=0,x_1,\ldots,x_N}}_{ y_o=0,y_1,\ldots,y_N} \prod
_{i=0}^{N-1}\mathbb{E} (\pi_{i,i+1}(x_i,x_{i+1})\pi
_{i,i+1}(y_i,y_{i+1}) )\nonumber\\
&&\hspace*{70pt}
{}\times\frac{e^{\langle\theta
,x_{i+1}-x_i\rangle}}{\phi(\theta)}\frac{e^{\langle\theta
,y_{i+1}-y_i\rangle}}{\phi(\theta)}\nonumber\\
&\hspace*{3pt}=&\mathop{\sum_{x_o=0,x_1,\ldots,x_N}}_{ y_o=0,y_1,\ldots,y_N}\prod
_{i=0}^{N-1}\frac{\mathbb{E} (\pi_{i,i+1}(x_i,x_{i+1})\pi
_{i,i+1}(y_i,y_{i+1}) )}{q(x_{i+1}-x_i)q(y_{i+1}-y_i)}\nonumber\\
&&\hspace*{70pt}
{}\times q^\theta
(x_{i+1}-x_i)q^\theta(y_{i+1}-y_i),\nonumber
\end{eqnarray}
where $q^\theta(z):=q(z)\frac{e^{\langle\theta,z\rangle}}{\phi
(\theta)}$ for every $z\in U$. For every $x\in\mathbb{Z}^d$, we let
$\hat{P}_x^\theta$ be the probability measure on paths starting at $x$
and induced by $ (q^\theta(z) )_{z\in U}$. We write $\hat{E}_x^\theta$
to denote expectation with respect to $\hat{P}_x^\theta$.

We note that $\mathbb{E} (\pi_{i,i+1}(x_i,x_{i+1})\pi
_{i,i+1}(y_i,y_{i+1}) )=q(x_{i+1}-x_i)q(y_{i+1}-y_i)$ unless $x_i=y_i$.
For every $x,y\in U$, let us set
\[
V(x,y):=\log\biggl(\frac{\mathbb{E} (\pi_{0,1}(0,x)\pi_{0,1}(0,y)
)}{q(x)q(y)} \biggr).
\]
By uniform ellipticity, $V$ is bounded by some constant $\bar{V}$. With
this notation,
\[
G_N(\theta)=\hat{E}_o^\theta\hat{E}_o^\theta\bigl[e^{\sum
_{i=0}^{N-1}\delta_{X_i=Y_i}V(X_{i+1}-X_i,Y_{i+1}-Y_i)} \bigr].
\]
We let $\tau:=\inf\{k\geq0\dvtx X_k=Y_k \}$, $\tau^+:=\inf\{k>0\dvtx X_k=Y_k \}$ and decompose $G_N(\theta)$ with respect to the first steps
$X_1$ and $Y_1$:
\begin{eqnarray*}
G_N(\theta)&=&\sum_{x,y}q^\theta(x)q^\theta(y)e^{V(x,y)}\sum
_{k=0}^{N-2}\hat{P}_x^\theta\hat{P}_y^\theta(\tau=k )G_{N-k-1}(\theta
)\\
&&{}+\sum_{x,y}q^\theta(x)q^\theta(y)e^{V(x,y)}\hat{P}_x^\theta
\hat{P}_y^\theta(\tau\geq N-1 )\\
&=&\sum_{k=0}^{N-2} \Biggl(\sum_{x,y}q^\theta(x)q^\theta(y)
e^{V(x,y)}\hat{P}_x^\theta\hat{P}_y^\theta(\tau=k ) \Biggr)G_{N-k-1}(\theta
)\\
&&{}+\sum_{x,y}q^\theta(x)q^\theta(y)e^{V(x,y)}\hat{P}_x^\theta
\hat{P}_y^\theta(\tau\geq N-1 ).
\end{eqnarray*}
We simplify the last expression by defining
\begin{eqnarray*}
B_k(\theta)&:=&\sum_{x,y}q^\theta(x)q^\theta(y)e^{V(x,y)}\hat
{P}_x^\theta\hat{P}_y^\theta(\tau=k ),\\
C_N(\theta)&:=&\sum
_{x,y}q^\theta(x)q^\theta(y)e^{V(x,y)}\hat{P}_x^\theta\hat
{P}_y^\theta(\tau\geq N-1 )
\end{eqnarray*}
and obtain the following equation:
%
\begin{equation}\label{neis}
G_N(\theta)=\sum_{k=0}^{N-2}B_k(\theta)G_{N-k-1}(\theta)+C_N(\theta).
\end{equation}
Now, we use the dimension. For every $x,y$ such that $x\neq y$, under
the product measure $\hat{P}_x^\theta\hat{P}_y^\theta$,
$(X_i-Y_i)_{i\geq0}$ is a symmetric random walk and since $d\geq3$, it
has positive probability of never hitting the origin. Therefore,
\[
\lim_{N\rightarrow\infty}C_N(\theta)=\inf_{N}C_N(\theta)
=\sum_{x,y}q^\theta(x)q^\theta(y)e^{V(x,y)}\hat{P}_x^\theta\hat
{P}_y^\theta(\tau=\infty)>0.
\]
By (\ref{aboo}), we know that $G_M(0)=1$ for every $M$. Plugging it in
(\ref{neis}), we get
\[
1=\sum_{k=0}^{N-2}B_k(0)+C_N(0).
\]
Taking $N\rightarrow\infty$ gives us
%
\begin{equation}\label{fisinial}
\sum_{k=0}^{\infty}B_k(0)<1.
\end{equation}

We would like to show that
\[
B(\theta):=\sum_{k=0}^{\infty}B_k(\theta)
\]
is continuous in $\theta$ at $0$. Since $\theta\mapsto B_k(\theta)$ is
continuous for each $k$, it suffices to argue that the tail of this sum
is small, uniformly in $\theta$ in a neighborhood of $0$. Indeed,
%
\begin{eqnarray}\label{faikcina}
\sum_{k=N}^{\infty}B_k(\theta)&\leq&e^{\bar{V}}\sum
_{x,y}q^\theta(x)q^\theta(y)\hat{P}_x^\theta\hat{P}_y^\theta(N\leq\tau
<\infty)\nonumber\\ &=&e^{\bar{V}}\hat{P}_o^\theta\hat
{P}_o^\theta(N+1\leq\tau^+<\infty)\nonumber\\
&\leq&e^{\bar
{V}}\sum_{k=N+1}^{\infty}\hat{P}_o^\theta\hat{P}_o^\theta(X_k=Y_k
).
\end{eqnarray}
Since $d\geq3$ and the covariance of $X_1-Y_1$ under $\hat{P}_o^\theta
\hat{P}_o^\theta$ is a nonsingular matrix whose entries are continuous
in $\theta$, the local CLT implies that the sum in (\ref{faikcina}) is
the tail of a series which converges uniformly in $\theta$ in a
neighborhood of $0$.

Now that we know $\theta\mapsto B(\theta)$ is continuous at $0$, we
recall (\ref{fisinial}) and see that there exists $\bar{\eta}>0$ such
that for every $\theta\in\mathbb{R}^d$ with $|\theta|<\bar{\eta}$,
$B(\theta)<1$. Letting $C(\theta):=\sup_{M}C_M(\theta)$, we turn to
(\ref{neis}) and conclude that
\[
\sup_{M\leq N}G_M(\theta)\leq\frac{C(\theta)}{1-B(\theta)}<\infty.
\]
Taking $N\rightarrow\infty$ gives the desired result.
\end{pf}

\subsection{Obtaining the function $u^\theta$}

From now on, we consider $d\geq3$ and $\theta$ as in Lemma \ref{eltu}.
For every $x\in\mathbb{Z}^d$ and $n,N\in\mathbb{Z}$ with $n<N$, we
recall (\ref{babalar}) and observe that $\mathbb{P}$-a.s.
\begin{eqnarray*}
u_N^\theta(\omega,n,x)&=&\frac{E_{n,x}^\omega[e^{\langle\theta
,X_N-X_n\rangle} ]}{\phi(\theta)^{N-n}}\\
&=&\sum_{y}\pi_{n,n+1}(x,y)e^{\langle\theta,y-x\rangle}\frac
{E_{n+1,y}^\omega[e^{\langle\theta,X_N-X_{n+1}\rangle} ]}{\phi
(\theta)^{N-n}}\\
&=&\sum_{y}\pi_{n,n+1}(x,y)e^{\langle\theta,y-x\rangle-\log\phi
(\theta)}u_N^\theta(\omega,n+1,y).
\end{eqnarray*}
$ (u_N^\theta(\cdot,n,x) )_{N>n}$ is a nonnegative martingale and
$\mathbb{P}$-a.s. converges to a limit $u^\theta(\cdot,n,x)$ which satisfies
%
\begin{equation}\label{denk}
u^\theta(\omega,n,x)=\sum_{y}\pi_{n,n+1}(x,y)e^{\langle\theta
,y-x\rangle-\log\phi(\theta)}u^\theta(\omega,n+1,y).
\end{equation}
By Lemma \ref{eltu}, $ (u_N^\theta(\cdot,n,x) )_{N>n}$ is uniformly
bounded in $L^2(\mathbb{P})$ and, therefore, the convergence takes
place also in $L^2(\mathbb{P})$.

\subsection{Some properties of $u^\theta$}

For every $x\in\mathbb{Z}^d$ and $n,N\in\mathbb{Z}$ with $n<N$, we know
by (\ref{budur}) that $ \|u_N^\theta(\cdot,n,x) \|_{L^1(\mathbb
{P})}=1$. Since $ (u_N^\theta(\cdot,n,x) )_{N>n}$ converges to $u^\theta
(\cdot,n,x)$ in $L^2(\mathbb{P})$, we immediately see that $ \|u^\theta
(\cdot,n,x) \|_{L^1(\mathbb{P})}=1$ and $u^\theta(\cdot,n,x)\in
L^2(\mathbb{P})$.

Next, we observe that $\mathbb{P}$-a.s.
\begin{eqnarray*}
u_N^\theta(T_{n,x}\omega,0,0)&=&\frac{E_{o,o}^{T_{n,x}\omega} [
e^{\langle\theta,X_N-X_o\rangle} ]}{\phi(\theta)^{N}}\\
&=&\frac
{E_{n,x}^\omega[e^{\langle\theta,X_{N+n}-X_n\rangle} ]}{\phi
(\theta)^{N}}\\
&=&u_{N+n}^\theta(\omega,n,x).
\end{eqnarray*}
Taking $N\to\infty$, we see that $\mathbb{P}$-a.s.
%
\begin{equation}\label{oldu}
u^\theta(T_{n,x}\omega,0,0)=u^\theta(\omega,n,x).
\end{equation}
We abbreviate the notation by setting
%
\begin{equation}\label{yeniu}
u^\theta(\omega):=u^\theta(\omega,0,0).
\end{equation}
Since $u_N^\theta(\cdot,0,0)$ is $\mathcal{B}_0^+$-measurable, it
follows that $u^\theta$ is $\mathcal{B}_0^+$-measurable as
well. 

Using (\ref{oldu}) and (\ref{yeniu}), we put (\ref{denk}) in the
following form: $\mathbb{P}$-a.s.
%
\begin{equation}\label{bariz}
u^\theta(\omega)=\sum_{z\in U}\overline{\pi}(\omega,T_{1,z}\omega
)e^{\langle\theta,z\rangle-\log\phi(\theta)}u^\theta
(T_{1,z}\omega).
\end{equation}

Finally, let us prove that $u^\theta>0$ holds $\mathbb{P}$-a.s. We
already know that $u^\theta\geq0$ holds $\mathbb{P}$-a.s. Clearly, (\ref
{bariz}) implies that $ \{\omega\dvtx u^\theta(\omega)=0 \}$ is invariant
under $T_{1,z}$ for every $z\in U$. Since the product environment
$\mathbb{P}$ is ergodic under shifts, $\mathbb{P}(u^\theta(\omega)=0)$
is either $0$ or $1$. But we know that $ \|u^\theta(\cdot,n,x) \|
_{L^1(\mathbb{P})}=1$ and, therefore, we conclude that $\mathbb
{P}(u^\theta(\omega)=0)=0$.

Now, we are ready to define a new transition kernel $\overline{\pi
}^\theta$ on $\Omega$ by a Doob $h$-transform: For every $z\in U$,
$\mathbb{P}$-a.s.
%
\begin{equation}\label{doobcan}
\overline{\pi}^\theta(\omega,T_{1,z}\omega):= \overline{\pi}(\omega
,T_{1,z}\omega)\frac{u^\theta(T_{1,z}\omega)}{u^\theta(\omega)}
e^{\langle\theta,z\rangle-\log\phi(\theta)}.
\end{equation}
$\overline{\pi}^\theta$ induces a probability measure $P_{k,x}^{\theta
,\omega}$ on particle paths starting at position $x$ at time $k$ and we
write $E_{k,x}^{\theta,\omega}$ to denote expectation under this measure.

\subsection{\texorpdfstring{Proofs of Theorems \protect\ref{AequalsQ} and
\protect\ref{strong}}{Proofs of Theorems 2 and 3}}

\mbox{}

\begin{pf*}{Proof of Theorem \ref{AequalsQ}}
For $d\geq3$ and $\bar{\eta}$ as in Lemma \ref{eltu}, we recall (\ref
{babalar}) and observe that if $|\theta|<\bar{\eta}$ then
%
\begin{equation}\label{higherz}
\qquad\quad\lim_{n\rightarrow\infty}\frac{1}{n}\log E_{o,o}^\omega\bigl[
e^{\langle\theta,X_n\rangle} \bigr]=\log\phi(\theta)+\lim_{n\rightarrow\infty
}\frac{1}{n}\log u_n^\theta(\omega,0,0)=\log\phi(\theta)
\end{equation}
because $\lim_{n\to\infty} u_n^\theta(\omega) = u^\theta(\omega)>0$
holds $\mathbb{P}$-a.s. Since $\log\phi$ is strictly convex and $\xi
_o=\nabla\log\phi(0)$,
\[
\{\nabla\log\phi(\theta)\dvtx |\theta|<\bar{\eta} \}
\]
is an open set containing the LLN velocity $\xi_o$. Hence, there exists
$\eta>0$ such that for every $\xi\in\mathcal{D}$ with $|\xi-\xi_o|<\eta
$ there is a unique $\theta$ satisfying $|\theta|<\bar{\eta}$ and $\xi
=\nabla\log\phi(\theta)$. Because $\theta\mapsto\log\phi(\theta)$ is
analytic, (\ref{higherz}) and the G\"{a}rtner--Ellis theorem (Dembo and
Zeitouni \cite{DZ}, page 44) immediately imply the desired
result.
\end{pf*}

\begin{pf*}{Proof of Theorem \ref{strong}}
Under the conditions of the theorem, we recall the proof of Theorem \ref
{averagedconditioning} and see that for the unique $\theta$ with $\xi
=\nabla\log\phi(\theta)$
\[
\limsup_{n\rightarrow\infty}\frac{1}{n}\log E_{o,o} \bigl[e^{\langle
\theta,X_n\rangle-n\log\phi(\theta)}, A_{\xi,n}^\epsilon(f) \bigr]=:\gamma<0.
\]
Fixing $\alpha>0$, for every $n\in\mathbb{N,}$ we define the events
\[
B_n':= \bigl\{\omega\dvtx E_{o,o}^\omega\bigl[e^{\langle\theta,X_n\rangle
-n\log\phi(\theta)}, A_{\xi,n}^\epsilon(f) \bigr]>e^{n(\gamma
+\alpha)} \bigr\}
\]
on $\Omega$. Then
\begin{eqnarray*}
\mathbb{P} (B_n' )&\leq&\int_{B_n'}E_{o,o}^\omega\bigl[e^{\langle
\theta,X_n\rangle-n\log\phi(\theta)}, A_{\xi,n}^\epsilon(f) \bigr]
e^{-n(\gamma+\alpha)}\,d\mathbb{P}\\
&\leq& E_{o,o} \bigl[e^{\langle\theta,X_n\rangle-n\log\phi(\theta
)}, A_{\xi,n}^\epsilon(f) \bigr]e^{-n(\gamma+\alpha)}.
\end{eqnarray*}
Therefore, $\limsup_{n\rightarrow\infty}\frac{1}{n}\log\mathbb{P} (B_n'
)\leq-\alpha,$ and in particular $\sum_{n=1}^{\infty}\mathbb{P} (B_n'
)<\infty$. By the Borel--Cantelli lemma, $\mathbb{P} (B_n' \mbox{ i.o.}
)=0$. In other words, $\mathbb{P}$-a.s.
\[
E_{o,o}^\omega\bigl[e^{\langle\theta,X_n\rangle-n\log\phi(\theta
)}, A_{\xi,n}^\epsilon(f) \bigr]\leq e^{n(\gamma+\alpha)}
\]
for sufficiently large $n$. Thus,
\[
\limsup_{n\rightarrow\infty}\frac{1}{n}\log E_{o,o}^\omega\bigl[
e^{\langle\theta,X_n\rangle-n\log\phi(\theta)}, A_{\xi,n}^\epsilon
(f) \bigr]\leq\gamma+\alpha.
\]
Since $\alpha>0$ is arbitrary, we actually see that for $\mathbb{P}$-a.e. $\omega$
\[
\limsup_{n\rightarrow\infty}\frac{1}{n}\log E_{o,o}^\omega\bigl[
e^{\langle\theta,X_n\rangle-n\log\phi(\theta)}, A_{\xi,n}^\epsilon
(f) \bigr]\leq\gamma.
\]

Let us now finish the proof of the theorem:
\begin{eqnarray*}
&&\limsup_{n\rightarrow\infty}\frac{1}{n}\log P_{o,o}^\omega(A_{\xi
,n}^\epsilon(f) |D_{\xi,n}^\delta )\\
&&\qquad\leq\limsup_{n\rightarrow\infty}\frac{1}{n}\log P_{o,o}^\omega(A_{\xi
,n}^\epsilon(f),D_{\xi,n}^\delta) - \liminf_{n\rightarrow\infty}\frac
{1}{n}\log P_{o,o}^\omega(D_{\xi,n}^\delta)\\
&&\qquad\leq\limsup_{n\rightarrow\infty}\frac{1}{n}\log E_{o,o}^\omega\biggl[
e^{\langle\theta,X_n\rangle},A_{\xi,n}^\epsilon(f), \biggl|\frac
{X_n}{n}-\xi\biggr|\leq\delta\biggr]\\
&&\quad\qquad{} - \langle\theta,\xi\rangle+ I_a(\xi) +
|\theta|\delta\\
&&\qquad\leq\limsup_{n\rightarrow\infty}\frac{1}{n}\log E_{o,o}^\omega\bigl[
e^{\langle\theta,X_n\rangle- n\log\phi(\theta)},A_{\xi,n}^\epsilon
(f) \bigr] + |\theta|\delta\\
&&\qquad\leq \gamma+ |\theta|\delta\\
&&\qquad< 0
\end{eqnarray*}
when $\delta>0$ is sufficiently small. In the above estimate, we use
the fact that the quenched LDP holds in a neighborhood of $\xi$ with
rate
\[
I_a(\xi)=\langle\theta,\xi\rangle-\log\phi(\theta)
\]
at $\xi$, which is true by hypothesis.
\end{pf*}

\section{Identifying $\mu_\xi^\infty$ as a stationary Markov
process}\label{Doobsection}

For $d\geq3$ and $|\xi-\xi_o|<\eta$ with $\eta$ as in Theorem \ref
{AequalsQ}, we let $\theta\in\mathbb{R}^d$ be the unique solution of
$\xi=\nabla\log\phi(\theta)$. We can put $\overline{\mu}_\xi^\infty$ in
a nicer form. For every $N,M$ and $K\in\mathbb{N}$ and any $f$ as in
Definition \ref{definemu}, setting $L:=N+M+K+1$, we have
%
\begin{eqnarray}\label{genc}
&&\int f\,d\overline{\mu}_\xi^\infty\nonumber\\
&&\qquad=E_{o,o} \bigl[e^{\langle\theta,X_{L}\rangle- L\log\phi(\theta
)}f(T_{N,X_N}\omega,(Z_{N+i})_{i\geq1}) \bigr]\nonumber\\
&&\qquad=\sum_{x}\mathbb{E} \bigl(E_{o,o}^\omega\bigl[e^{\langle\theta
,X_{L}\rangle- L\log\phi(\theta)}f(T_{N,X_N}\omega,(Z_{N+i})_{i\geq
1}),X_L=x \bigr] \bigr)\nonumber\\
&&\quad\qquad\phantom{\sum_{x}}   {}\times\mathbb{E} (u^\theta(T_{L,x}\omega) )\nonumber\\
&&\qquad=\sum_{x}\mathbb{E} \bigl(E_{o,o}^\omega\bigl[e^{\langle\theta
,X_{L}\rangle- L\log\phi(\theta)}u^\theta(T_{L,x}\omega
)f(T_{N,X_N}\omega,(Z_{N+i})_{i\geq1}),X_L=x \bigr] \bigr)\nonumber\\
&&\qquad=\mathbb{E} \biggl(u^\theta(\omega)E_{o,o}^\omega\biggl[e^{\langle\theta
,X_{L}\rangle- L\log\phi(\theta)}\frac{u^\theta(T_{L,X_L}\omega
)}{u^\theta(\omega)}f(T_{N,X_N}\omega,(Z_{N+i})_{i\geq1}) \biggr] \biggr)\nonumber\\
&&\qquad=\mathbb{E} (u^\theta(\omega)E_{o,o}^{\theta,\omega} [f(T_{N,X_N}\omega
,(Z_{N+i})_{i\geq1}) ] ),
\end{eqnarray}
where we use the facts that $\mathbb{E} (u^\theta(T_{L,x}\cdot) )=1$
and $u^\theta(T_{L,x}\cdot)$ is $\mathcal{B}_L^+$-measurable.

We note that (\ref{genc}) is independent of $M$ and $K$. This
immediately tells us that the marginal $\mu_\xi^1$ of $\mu_\xi^\infty$
is absolutely continuous relative to $\mathbb{P}$ on every $\mathcal
{B}_{-N}^+$. Here is how: We fix $N\in\mathbb{N}$. For any $M\in\mathbb
{N}$ and any bounded $\mathcal{B}_{-N}^+\cap\mathcal
{B}_{M}^-$-measurable $h\dvtx\Omega\rightarrow\mathbb{R}$, we have
\begin{eqnarray*}
\int h\,d\mu_\xi^1&=&\mathbb{E} (u^\theta(\omega)E_{o,o}^{\theta
,\omega} [h(T_{N,X_N}\omega) ] )\\
&\leq&\|u^\theta\|_{L^2(\mathbb{P})} \|E_{o,o}^{\theta,\omega}
[h(T_{N,X_N}\omega) ] \|_{L^2(\mathbb{P})}\\
&\leq&\|u^\theta\|_{L^2(\mathbb{P})} \Biggl\|\sum_{|x|\leq N}|h(T_{N,x}\omega
) |\Biggr\|_{L^2(\mathbb{P})}\\
&\leq&(2N+1)^d \|u^\theta\|_{L^2(\mathbb{P})} \|h \|_{L^2(\mathbb{P})}.
\end{eqnarray*}
Since such functions are dense in $L^2(\Omega,\mathcal{B}_{-N}^+,\mathbb
{P})$, it follows by the Riesz representation theorem that
%
\begin{equation}\label{riesz}
\frac{d\mu_{\xi}^1}{d\mathbb{P}} \biggm|_{\mathcal
{B}_{-N}^+}\in L^2(\mathbb{P}).
\end{equation}

\begin{pf*}{Proof of Theorem \ref{doob}}
We have obtained $u^\theta$ in (\ref{yeniu}) and defined
$\overline{\pi}^\theta$ in (\ref{doobcan}). For every $N,K\in\mathbb
{N}$, we take any bounded $f:\Omega^{K+1}\rightarrow\mathbb{R}$ and
$g:\Omega^{\mathbb{N}}\rightarrow\mathbb{R}$ such that
\[
f (\omega,T_{1,z_1}\omega,\ldots,T_{K,z_1+\cdots+z_K}\omega)g
(T_{K,z_1+\cdots+z_K}\omega,T_{K+1,z_1+\cdots+z_{K+1}}\omega,\ldots)
\]
is $\mathcal{B}_{-N}^+$-measurable for any $(z_i)_{i\geq1}$. Then
\begin{eqnarray*}
&&\int f(\omega_1,\ldots,\omega_{K+1})g(\omega_{K+1},\omega_{K+2},\ldots
)\,d\mu_\xi^\infty\\
&&\qquad=\int f (\omega,T_{1,z_1}\omega,\ldots,T_{K,z_1+\cdots+z_K}\omega)\\
&&\quad\qquad\phantom{\int}
{}\times g(T_{K,z_1+\cdots+z_K}\omega,T_{K+1,z_1+\cdots+z_{K+1}}\omega,\ldots
)\,d\overline{\mu}_\xi^\infty\\
&&\qquad=\mathbb{E} (u^\theta(\omega)E_{o,o}^{\theta,\omega} [f(T_{N,X_N}\omega
,\ldots,T_{N+K,X_{N+K}}\omega)g(T_{N+K,X_{N+K}}\omega,\ldots) ] )\\
&&\qquad=\mathbb{E} (u^\theta(\omega)E_{o,o}^{\theta,\omega} [f(T_{N,X_N}\omega
,\ldots,T_{N+K,X_{N+K}}\omega)\\
&&\phantom{\qquad=\mathbb{E} (u^\theta(\omega)E_{o,o}^{\theta,\omega} [}
{}\times E_{N+K,X_{N+K}}^{\theta,\omega}
[g(T_{N+K,X_{N+K}}\omega,\ldots) ] ] )\\
&&\qquad=\int f(\omega_1,\ldots,\omega_{K+1})E_{o,o}^{\theta,\omega_{K+1}}
[g(\omega_{K+1},T_{1,X_1}\omega_{K+1},\ldots) ]\,d\mu_\xi^\infty
\end{eqnarray*}
where we use (\ref{genc}) and the Markov property. This proves that $\mu
_\xi^\infty$ is indeed a Markov process with state space $\Omega$ and
transition kernel $\overline{\pi}^\theta$.

We already know that $\mu_\xi^\infty$ is a stationary process. Hence,
its marginal $\mu_\xi^1$ is an invariant measure for $\overline{\pi
}^\theta$. Since $\mu_\xi^1$ is absolutely continuous relative to
$\mathbb{P}$ on every $\mathcal{B}_{-N}^+$ [by (\ref{riesz})], it
follows that $\mu_\xi^1$ is the unique invariant measure for $\overline
{\pi}^\theta$ with that absolute continuity property (see Rassoul-Agha
\cite{Firas}).
\end{pf*}

\section*{Acknowledgments}
This work is part of my Ph.D. thesis. I am grateful to my advisor
S. R. S. Varadhan for generously sharing his ideas and patiently
guiding me throughout my studies. I also thank F. Rassoul-Agha and
T. Sepp\"al\"ainen for valuable remarks and suggestions.

%

\printaddresses

\end{document}